\documentclass[11pt]{article}
\usepackage{amsfonts}
\usepackage{makeidx}
\usepackage{amsthm,amssymb}
\usepackage{latexsym}
\usepackage{graphicx}
\usepackage{amsmath}

\usepackage{enumerate}
\usepackage[T1]{fontenc}
\usepackage{amsfonts}
\usepackage[normalem]{ulem}

\newtheorem{theorem}{Theorem}

\newtheorem{adefinition}[theorem]{Definition}

\newtheorem{aexample}[theorem]{Example}

\newtheorem{lemma}[theorem]{Lemma}

\newtheorem{aremark}[theorem]{Remark}
\newenvironment{remark}{\begin{aremark}\rm}{\end{aremark}}

\numberwithin{equation}{section} \numberwithin{theorem}{section}
\DeclareMathOperator\Tr{Tr}

\newcommand{\p}{\mathbf{P}}

\newcommand{\N}{\mathbb{N}}
\newcommand{\NN}{N}
\newcommand{\M}{\mathcal{M}}

\newcommand{\E}{\mathbf{E}}
\newcommand{\V}{\mathbf{Var}}
\newcommand{\R}{\mathbb{R}}
\newcommand{\C}{\mathbb{C}}

\newcommand{\Qa}{Q_\alpha}
\newcommand{\Qb}{Q_\beta}

\newcommand{\Ak}{ A_{k\ell}}
\newcommand{\Ai}{ A_{ij}}
\newcommand{\Qk}{Q^{k\ell}}
\newcommand{\Qn}{Q_{(n)}}
\newcommand{\Gk}{ G^{k\ell}}
\newcommand{\Gi}{ G^{ij}}
\newcommand{\Yk}{ Y^{k\ell}}
\newcommand{\Yi}{ Y^{ij}}
\newcommand{\al}{\alpha}
\newcommand{\be}{\beta}
\newcommand{\A}{\mathcal{A}_{rd}}
\newcommand{\LL}{\mathcal{L}_{rd}}

\newcommand{\vv}{\mathbf{v}}
\newcommand{\uu}{\mathbf{u}}

\newcommand{\ya}{\mathbf{y}_\alpha}
\newcommand{\xa}{\mathbf{x}_\alpha}

\newcommand{\xb}{\mathbf{x}_\beta}
\newcommand{\yb}{\mathbf{y}_\beta}
\newcommand{\fa}{f_{n, \al}}
\newcommand{\am}{ \al}
\newcommand{\bm}{ \beta}
\newcommand{\ga}{ \gamma}
\newcommand{\ws}{\widetilde{\sigma}}
\newcommand{\D}{\mathbf{\delta}}

\makeatother

\begin{document}

\title{On the empirical spectral distribution
for certain models
\\
related to sample covariance matrices with different correlations}
\author{Alicja Dembczak-Kołodziejczyk, Anna Lytova
}
\newcommand\address{\noindent\leavevmode

\noindent
Alicja Dembczak-Kołodziejczyk,\\
University of Opole,\\
48 Oleska, 45-052,\\
Opole, Poland.\\
\texttt{\small
e-mail:
alicja.dembczak@uni.opole.pl
}
\medskip

\noindent
Anna Lytova,\\
University of Opole,\\
48 Oleska, 45-052,\\
Opole, Poland.\\
\texttt{\small
e-mail:
alytova@uni.opole.pl
}
}
\date{}
\maketitle

\begin{abstract}
Given $n,m\in \N$, we study two classes of large random matrices of the form
$$
\mathcal{L}_n   =\sum_{\al=1}^m\xi_\alpha \ya  \ya^T\quad\text{and}\quad \mathcal{A}_n   =\sum_{\al=1}^m\xi_\alpha (\ya  \xa^T+\xa  \ya^T),
$$
where for every $n$, $(\xi_\al)_\al\subset \mathbb{R}$ are iid  random variables independent of $(\xa,\ya)_\al$, and $(\xa)_\al$, $(\ya)_\al\subset \R^n$ are two (not necessarily independent) sets of independent random vectors having {\it different covariance matrices} and generating well concentrated bilinear forms. We consider two main asymptotic regimes as $n,m(n)\to \infty$:  a standard one, where $m/n\to c$, and a slightly modified one, where $m/n\to\infty$ and $\E\xi\to 0$  while $m\E\xi /n\to c$ for some $c\ge 0$. Assuming that vectors  $(\xa)_\al$ and $(\ya)_\al$ are normalized and isotropic ``in average'', we prove the convergence in probability of the empirical spectral distributions of $\mathcal{L}_n   $ and $\mathcal{A}_n   $ to a version of the Marchenko-Pastur law and so called effective medium spectral distribution, correspondingly. In particular, choosing normalized Rademacher random variables as $(\xi_\al)_\al$, in the modified regime one can get  a shifted semicircle and semicircle laws. We also apply our results to the certain classes of  matrices having block structures, which were studied in \cite{Cic:18, Cic:19}.
\end{abstract}

\section{Introduction}
\label{s:intro}
In \cite{Cic:18, Cic:19}, the authors studied the empirical spectral distributions of the following two related models of sparse block matrices.
Given $r,d\in \N$, let $(\vv^{kl})_{1\le k<l\le r}$ be 
independent copies of a random vector $\vv$ uniformly distributed on the unit sphere in $\R^{d}$, and let  $(\xi_{kl})_{1\le k<l\le r}$ be independent copies of a $0/1$ random variable $\xi=\xi_r$ such that $\E \xi=p_r$ for some $p_r\in(0,1]$. Define $\A$ and  $\LL$ as $rd\times rd$ block matrices of the form
\begin{equation}\label{AL}
  \A=\Big((1-\delta_{k\ell })B_{k\ell }\Big)_{k,\ell=1}^r,\quad \LL=\Big(\delta_{k\ell}\sum _{j\neq k }B_{j\ell}\Big)_{k,\ell=1}^r-\A,
\end{equation}
where for $1\le k<l\le r$ blocks
\begin{equation}\label{B}
  B_{k\ell}=\xi_{k\ell}\vv^{k\ell}\vv^{k\ell\,T}
\end{equation}
are $d\times d$ rank-one matrices with probability $p_r$ (and 0 otherwise).
These models were introduced in \cite{Cic:18} while studying the elastic
vibrational modes of amorphous solids. Roughly speaking they describe a system of $r$ $d$-dimensional points connected
by springs (see also \cite{LM:06,ZS:11} and references therein for the application of these models in
the study of certain disordered systems).  Evidently, for $d=1$  matrices $\A$ and  $\LL $ reduce to the adjacency matrix and  Laplacian of the Erd\H{o}s--R\'{e}nyi graph.

Recall that given a Hermitian or symmetric $n\times n$ matrix $\M_n$ with eigenvalues $(\lambda_i)_i$, the empirical spectral distribution $\NN _{\M_n}$ of $\M_n$ and its Stieltjes transform $s_{\M_n}$ are  defined by the formulas
$$
\NN _{\M_n}(\Delta)=\vert\{i:\,\lambda_i\in\Delta\}\vert/n,\,\, \forall\Delta\subset\R,\quad\text{and}
$$
$$
s_{\M_n}(z)=\int\frac{ \NN _{\M_n}(d\lambda)}{\lambda-z}=\frac{1}{n}\Tr(\M_n-z)^{-1},\,\, \Im z\neq 0.
$$

In \cite{Cic:18, Cic:19},  the authors studied the empirical spectral distributions $\NN _{\A}$ and $\NN _{\LL}$ as $r\rightarrow\infty$ in different asymptotic regimes, depending on $d$ and $p_r$. In the case when $d$ and $rp_r$ are some fixed numbers (sparse matrices), the first several moments of the limiting  distributions were computed. In the ``dense'' regime when
$$
d\rightarrow\infty, \quad p_r=O(1),\quad\text{and}\quad p_r r/d\rightarrow c >0, \quad\text{as}\quad r\rightarrow\infty,
$$
the {\it convergence in mean} of $\NN _{\LL}$ and $\NN _{\A}$ to the Marchenko-Pastur law and to the so called effective medium spectral distribution, correspondingly,  was proved. (Here we follow terminology from \cite{Cic:18}, see also \cite{SC:02}.) It was shown that the limits of the corresponding  Stieltjes transforms, $f_{\mathcal{L}}:=\lim_{r\rightarrow \infty}\E s_{\LL}$ and
$f_{\mathcal{A}}:=\lim_{r\rightarrow \infty}\E s_{\A}$, satisfy equations
\begin{align}\label{fAfL}
 2zf_{\mathcal{L}}^2+(z+2-c )f_{\mathcal{L}}+1=0\quad\text{and}\quad zf_{\mathcal{A}}^3+(1-c )f_{\mathcal{A}}^2-zf_{\mathcal{A}}-1=0.
\end{align}
 Also the first several moments of $\NN _{\A}$ and  $\NN _{\LL}$ were calculated in the so called ``dilute'' regime  when  $d$ is fixed and $p_r=O(r^{-\delta})$, $\delta\in(0,1)$, as $r\rightarrow\infty$, and it was claimed that for big enough $r$ these moments coincide with the moments of  the semicircle and shifted semicircle distributions and, in particular,
\begin{align}\label{cr}
 2f_{\LL}^2+(z-c_r)f_{\LL}+1=0\quad\text{and}\quad c_rf_{\A}^2+zf_{\A}+1=0,
\end{align}
where $c_r:=rp_r/d\to\infty$, as $r \to\infty$.

In our work we suppose that $d$ grows to infinity with $r$. We modify the dilute regime as follows:
$p_r\to 0$ and $r/d\to \infty$ while $c_r=rp_r/d\to c  \ge0$ as  $r,d\to\infty$. This guarantees that the corresponding sequences of empirical spectral distributions $\NN _{\LL}$ and $\NN _{\A}$ are tight and also allows to unify two regimes as follows:
\begin{equation}\label{reg}
  d\to \infty \quad\text{and}\quad  rp_r/d\to c  \ge0 \quad\text{as}\quad  r\to \infty.
\end{equation}
  We consider models (\ref{AL}) corresponding to normalized isotropic vectors $(\vv^{k\ell})_{k<\ell}$ which  generate well-concentrated bilinear forms (see Assumption 1), and applying the Stieltjes transform method give a straightforward proof of the convergence in probability of  $\NN _{\LL}$ and $\NN _{\A}$ to the Marchenko-Pastur law and the effective medium spectral distribution. We show that the limits are the same for both regimes (in contrast to (\ref{fAfL}) and (\ref{cr})), and that to get (\ref{cr}) with $c_r=c $ in the new dilute regime,  $(\xi^{k\ell})_{k<\ell}$ needs
  to take both negative and positive values with non-zero probability. In particular, one can  get (\ref{cr}) if   $(\xi^{k\ell})_{k<\ell}$ are properly normalized Rademacher random variables (see Example 3 and Remark \ref{r:R}).

 Note that we can rewrite matrices $\A$ and $\LL$ in the form
\begin{equation}\label{LY}
   \LL=\sum_{1\le k < \ell\le r}\xi_{k\ell} Y^{k\ell}Y^{k\ell\,T}\quad\text{and}\quad
   \A=\sum_{1\le k \neq \ell\le r}\xi_{k\ell} X^{k\ell}X^{\ell k\,T},
\end{equation}
where $\xi_{k\ell}=\xi_{\ell k}$, $(X^{k\ell})_{k\neq\ell}$ and $(Y^{k\ell})_{k<\ell}$ are sparse block vectors in $\R^{rd}$ given by
$$
Y^{k\ell}=((\delta_{jk}-\delta_{j\ell})\vv^{kl})_{j=1}^{r}\quad\text{and}\quad X^{k\ell}=(\delta_{jk}\vv^{kl})_{j=1}^{r}.
$$
This naturally leads to the study of  more general classes of random matrices of the form
\begin{equation}\label{AL1}
\mathcal{L}_n   =\sum_\am\xi_\alpha \ya  \ya^T\quad\text{and}\quad \mathcal{A}_n   =\sum_\am\xi_\alpha (\ya  \xa^T+\xa  \ya^T),
\end{equation}
 where $\xi_\al\in\R$ and $\ya, \xa \in\R^n$, $\al\le m$,  are some random variables and vectors.
 We mainly concentrate on model $\mathcal{L}_n   $, which is closely related to the sample covariance matrices, only that here we allow  vectors $(\ya )_\al$ to have {\it different  covariance matrices} $\Qa:=\E \ya \ya^T$, $\al\le m$ (note that here $\Qa$ are not necessarily centered). We suppose that these vectors are normalized and isotropic "in average", $m^{-1}\sum_\al \Qa\approx n^{-1}I_n$, which allows  to show that the empirical spectral distributions still converge to the Marchenko-Pastur law.
A similar model was considered in \cite{Yin:18} and \cite{L:18}, where the convergence of spectral distributions was studied, in particular, for matrices of the form $\sum_\am  \ya\ya^T $ corresponding to vectors with essentially different covariance matrices (not isotropic in average). In these papers the limiting distribution is  given implicitly (in terms of asymptotic closeness to the solution of a certain system of equations), and our result does not follow directly from  \cite{Yin:18, L:18}.
Certain closely related models were also studied in \cite{BVZ:19}, where the authors proved the convergence to the Marchenko-Pastur law of empirical spectral measures corresponding to the certain block-independent models and tensor models (in \cite{BVZ:19}, see also a review of known results on convergence to the Marchenko-Pastur law with relaxed independence requirements including \cite{Ma-Pa:67, YK:86, GT:05, Au:06, Ba-Zh:08, PP:09, Adam:11, R:12, G-N-T:14,Y:16}).
\medskip

In (\ref{LY}), we choose $(\ya)_\al$ and  $(\xi_\al)_\al$  from the following classes.
\medskip

\noindent {\bf Assumption 1.} {\it We suppose that for every $n\in \N$,  $\ya=\mathbf{y}_{\al,n}\in\R^n$, $\al\le m, $ are mutually independent random vectors such that  for all deterministic matrices $D=D_n$ with $\|D\|_{op}=1$ we have 
    \begin{equation}\label{Ayy2}
    \sup_\al\V(D\, \ya,\ya)
      ={o}(1),\quad n\rightarrow\infty.
    \end{equation}
}Here and in what follows, given a matrix $D$ we use notations $\|D\|_{op}$ and $\|D\|_{HS}$ for its operator and Hilbert-Schmidt norms.
\medskip

\noindent {\bf Assumption 2.} {\it For every $n\in \N$,  let $\xi_\alpha=\xi_{\alpha,n}\in\R$, $\al\le m, $ be mutually independent copies of a random variable $\xi_n$ with a cumulative distribution function $\sigma_n$. To treat simultaneously both cases, $m/n=O(1)$ and $m/n\to \infty$, $\frac{m}{n}\E\xi_n=O(1)$, we introduce a signed measure $\ws_n$, which controls $\frac{m}{n}\E\xi_n$.  Let $\ws_n$ be defined as follows: for every finite $\Delta\subset\R$
$$
\ws_n(\Delta)=\frac{m}{n}\int_\Delta \xi d\sigma_n(\xi).
$$
We suppose that as $n\to\infty$, $\ws_n$ converges weakly to a signed measure $\ws$ such that $|\ws(\R)|<\infty$, and
  \begin{align*}
  &\sup_{n}\int |\xi^{p}d\ws_n(\xi)|=\sup_{n}\frac{m}{n}\E|\xi_n|^{p+1}<\infty,\quad p=1,2,3.
      \end{align*}
      }
      We use notation $c_1:=\ws(\R)\in\R$, $|c_1|<\infty$.
      
\begin{remark}{\it
  Note that if $m/n\to c>0$, then $\ws_n=c\sigma_n+o(1)$ as $n\to\infty$. Also, a bit more delicate but quite standard nowadays argument based on a truncation procedure for $\xi_n$ (see, for example, \cite{PP:09}) allows to show that the results of Theorem \ref{t:1} below remain valid without any moment conditions on $\sigma_n$ in the case $m/n=O(1)$ and with the only moment condition $\frac{m}{n}\E\xi_n=O(1)$ in the case $m/n\to \infty$,  $m/n^2\to 0$.}
\end{remark}

Our main result concerns  convergence of the empirical spectral distributions of $(\mathcal{L}_n)_n$, it can be considered as a generalization of Theorem 3.3 of \cite{PP:09} (for $H^{(0)}=0$) on the case of ``samples'' with entries having different covariance matrices.

\begin{theorem}\label{t:1}
  Given $m,n\in \N$, consider $n\times n$ matrices
  $$\mathcal{L}_n   =\sum_\am \xi_\alpha \ya  \ya^T ,$$
  where $\xi_\alpha$ and $\ya\in\R^n$, $\alpha\le m$,  are mutually independent random variables satisfying Assumptions 1, 2.   Let
  $
  \Qa :=\E \ya \ya^T
  $
  be such that
  \begin{align}
    &\sup_{\al} \|\Qa\|_{op}=O(n^{-1}),
    \quad  \sup_{\al}|\Tr \Qa-1|=o(1),\,\,\text{and}\label{Qal}
    \\
    & \Qn  :=\frac{1}{m}\sum_\am \Qa=\frac{1}{n}I_n+B_n,\,\,\text{where}\,\,\|B_n\|_{HS}=o(n^{-1/2}),\,\,n\to\infty.\label{Q}
  \end{align}
Then as $n\to\infty$   the empirical spectral distributions $\NN _{\mathcal{L}_n   }$ converge in probability to a non-random probability measure $\NN _{\mathcal{L}}$  which  Stieltjes transform $f$ is uniquely determined by the equation
  \begin{equation}\label{eq:I}
    zf(z)=-1+f(z)\int\frac{d\ws(\xi)}{1+\xi f(z)}
  \end{equation}
in the class of Stieltjes transforms of non-negative measures.
   \end{theorem}
\begin{remark}\label{r:a}
  A simple renormalization allows to show that if in (\ref{Qal}) and (\ref{Q}) we have
  $$
   \sup_{\al}|\Tr \Qa-a|=o(1)\quad\text{and}
    \quad
     \Qn =\frac{a}{n}I_n+B_n\quad\text{for some}\,\, a>0,
  $$
  then
 $
    zf=-1+af\int(1+a\xi f)^{-1}d\ws(\xi).
 $
Also in the case $d\ws(\xi)=c\xi d\sigma(\xi)$ we restore the Marchenko-Pastur distribution.
\end{remark}

Some additional information about the moments of $\xi_\al$ allows to solve (\ref{eq:I}) exactly, here are several simple examples:
\medskip

\noindent{\bf Examples.} In the following three examples we use notations $\D$ and $\D'$ for the Dirac delta function and its generalized derivative, and $\rho$, $\widetilde{\nu}_n$, $\widetilde{\nu}$ for the densities of $\NN$, $\ws_n$, $\ws$, correspondingly.

1. Let $(\xi_n)_n$ be not random, and $\xi_{\al,n}=b_n\to b$, $mb_n/n\to c_1$ as $n\to\infty$. Then
$$\widetilde{\nu}_n=\frac{m}{n}\xi\D(\xi-b_n)\to\widetilde{\nu}=c_1\D(\xi-b),$$ and by (\ref{eq:I}) $f$ satisfies $bzf^2+f(z+b-c_1)+1=0$, so that
\begin{equation}\label{ex1}
  \rho(\lambda)=\left\{
\begin{array}{cc}
\D(\lambda-c_1)&\quad \text{if}\quad b=0,
\\
\frac{1}{2\pi b\lambda}\sqrt{((c^+-\lambda)(\lambda-c^-))_+},&\quad  \text{if}\quad b\neq0,
\end{array}
\right.
\end{equation}
where $x_+=x$ if $x\ge 0$ and 0 otherwise and $c^\pm=(\sqrt{b}\pm \sqrt{c_1})^2$.

2. Let for every $n$, $(\xi_n)_n$ are 0/1 random variables such that $\p( \xi_n=1)m/n\to c_1$ as $n\to\infty$. Then again
$\widetilde{\nu}=c_1\D(\xi-1)$, and we get (\ref{ex1}) with $b=1$.

3. Suppose that all moments of $\widetilde{\sigma}$ are finite,
$$c_j:=\int\xi^{j-1}d\widetilde{\sigma}(\xi)=\lim\limits_{n\to\infty}\frac{m}{n}\E \xi_n^j<\infty, \quad \forall{j\geq1},$$
and for some $k_0\geq 1$ we have $c_j=0$ $\forall{j>k_0}$. Note that this is possible only if $k_0\le 2$, and moreover, for $k_0=2$ the condition
 $0<c_2<\infty$ while $c_3=0$ is not fulfilled for pure non-negative (or pure non-positive) random variables.
Indeed, if $\xi_n\geq 0$ a.s. then by the Schwartz inequality we would have
$$0<c_2=\lim\frac{m}{n}E\xi_n^2\leq \lim \frac{m}{n}(E\xi_n^3E\xi_n)^{1/2}=c_1c_3=0.
$$
So let $\xi_n$ take both negative and positive  values with positive probability, and $c_j=0$ $\forall{j\ge 3}$. Then expanding $(1+\xi f(z))^{-1}$ into the Taylor's series we get from (\ref{eq:I})
$c_2f^2+(c_1-z)f+1=0$,
thus in this case the limiting density is given by the shifted semicircle law,
$$
\rho(\lambda)=\frac{1}{2\pi c_2}\sqrt{(4c_2-(\lambda-c_1)^2)_+}.
$$
For example, if $(\xi_n)_n$  take values $\pm \sqrt{{n}/{m}}$ with probability ${1}/{2}$, than $\widetilde{\nu}_n\to\widetilde{\nu}=\D'$,
$c_2=1$, $c_j=0$, $j\neq 2$, and  $\rho(\lambda)=\frac{1}{2\pi}\sqrt{(4-\lambda^2)_+}.$

\medskip

\medskip

Return now  to the starting point of this research, namely, models $\LL$ and $\A$ (\ref{AL}) introduced and studied in \cite{Cic:18,Cic:19}, and consider first Laplacian $\LL$. It is easy to check that the condition $\|\Qa\|_{op}=O(n^{-1})$ of Theorem \ref{t:1} is not fulfilled (now $n=rd$ while $\|\E Y^{k\ell} Y^{k\ell T}\|_{op}=O(d^{-1})$), hence we cannot  apply  Theorem \ref{t:1} directly. Nevertheless, using the sparsity of vectors $\ya=Y^{k\ell}$ and slightly modifying the proof of Theorem \ref{t:1}, we get the following result for $\LL$:

\begin{theorem}\label{t:2}
  Let $\LL$ be defined in (\ref{AL}) - (\ref{B}), where for every $r\in \N$, $(\xi_{k\ell})_{1\le k<l\le r}$ are  iid  copies of a $0/1$ random variable $\xi=\xi_r$ with $\p(\xi=1)=p_r$, and $(\vv^{kl})_{1\le k<l\le r}$ are mutually independent normalized  isotropic random vectors, $\E\vv^{kl}\vv^{kl T}=d^{-1}I_d$,  satisfying Assumption 1 and having norms uniformly bounded in $r$. 
  Then in regime (\ref{reg}), $\NN _{\LL}$ converge in probability to a non-random probability measure $\NN _{\mathcal{L}}$  
  with the density
$$
\rho(\lambda)=\frac{1}{4\pi \lambda}\sqrt{((c^+-\lambda)(\lambda-c^-))_+},\quad c^\pm=(\sqrt{2}\pm \sqrt{c })^2.
$$
\end{theorem}
\medskip

As to adjacency matrices $\A$, in Section \ref{s:t4} we first treat matrices having a more general structure and prove an analog of Theorem \ref{t:1} for  matrix $\mathcal{A}_n$ defined in (\ref{AL1}) (see Theorem \ref{t:A}). Then using essentially the same scheme we get the following result for $\A$:

\begin{theorem}\label{t:A}
 Let $\A$ be defined in (\ref{AL}) - (\ref{B}), where for every $r\in \N$, $(\xi_{k\ell})_{1\le k<l\le r}$ are  iid  copies of a $0/1$ random variable $\xi=\xi_r$ with $\p(\xi=1)=p_r$, and $(\vv^{kl})_{1\le k<l\le r}$ are mutually independent normalized  isotropic random vectors, $\E\vv^{kl}\vv^{kl T}=d^{-1}I_d$,  satisfying Assumption 1 and having norms uniformly bounded in $r$.
 Then in regime (\ref{reg}), $\NN _{\A}$ converge in probability to a non-random probability measure $\NN_{\mathcal{A}}$  which  Stieltjes transform $f_{\mathcal{A}}$ is uniquely determined by the second equation in (\ref{fAfL})
  in the class of Stieltjes transforms of non-negative measures.
\end{theorem}

\begin{remark}\label{r:R}
1. One can find the explicit forms of the solution of the cubic equation (\ref{fAfL}) and the density of $\NN _{\mathcal{A}}$  in \cite{SC:02} and \cite{Cic:18}.

2. It can be shown that if $(\xi^{k\ell})_{k\neq \ell}$  take values $\pm \sqrt{{d}/{r}}$ with probability ${1}/{2}$, then
$f_{\mathcal{A}}$ solves equation
$f_{\mathcal{A}}^2+zf_{\mathcal{A}}+1=0$ (cf (\ref{cr})), so that the limiting density is given by the semicircle law $\rho(\lambda)=\frac{1}{2\pi}\sqrt{(4-\lambda^2)_+}$ (see also Remark \ref{r:A}.)
\end{remark}

The structure of the remaining part of the paper is very simple: in Sections \ref{s:t1}, \ref{s:t2}, and \ref{s:t4} we give the proofs of Theorems \ref{t:1}, \ref{t:2}, and \ref{t:A}, correspondingly. The proof of Theorem \ref{t:1} (based on \cite{PP:09}) is more detailed, while in the rest of the proofs we mostly discuss places which should be modified.


\bigskip

{\bf Acknowledgments} A.L. was supported by grant
nr 2018/31/B/ST1/03937 National Science Centre, Poland. A.L. also would like to thank the organizers of  XV Brunel – Bielefeld Workshop on Random Matrix Theory and Applications for the excellent conditions  and Prof. Cicuta for the introducing to the problem during this workshop.

\section{Proof of Theorem \ref{t:1}}
\label{s:t1}

The proof is based on the standard nowadays method of Stieltjes transform  which goes back to \cite{Ma-Pa:67} (see \cite{Ak-Gl:93,AGZ:10, BS:10,Pa-Sh:11} for the details of the method and main properties of the Stieltjes transform), and which is used in a huge number of results on convergence of empirical spectral distributions of random matrices. This method is based on the fact that there is a one-to-one continuous correspondence between non-negative measures and their Stieltjes transforms, so that to find a weak limit in probability  of random probability measures $\NN _{\mathcal{L}_n   }$  it is enough to show that for every $z\in \C\setminus\R$ the Stieltjes transforms $s_n:=s_{\mathcal{L}_n}$ of $\NN _{\mathcal{L}_n}$
converge in probability to a deterministic limit $f$ satisfying  $\lim_{\eta\to \infty}\eta|f(i\eta)|=1$. Then $f$ is the Stieltjes transform of a probability measure $\NN$ such that $\NN _{\mathcal{L}_n   }$ converge weakly in probability to $\NN$ and for every $\Delta\subset\R$
$$
\NN(\Delta)=\frac{1}{\pi}\lim_{\eta\to +0}\int_\Delta f(\lambda+i\eta)d\lambda.
$$
Our scheme of the proof is as follows: in Lemma \ref{l:var} we show that $\V s_n(z)=o(1)$ as $n\to\infty$, that reduces the problem to finding the limit of the expectations $\E s_n:=f_n$, then in the main body of the proof (Lemma \ref{l:mean}) we show  that for every convergent subsequence of $(f_n)_n$, its limit satisfies (\ref{eq:I}), and finally, the unique solvability  of (\ref{eq:I}) in  the class of the Stieltjes transforms of probability measures follows from Lemma (\ref{l:solv}) below.

\begin{lemma}\label{l:solv} (Solvability and uniqueness).
 Let $\ws$ be a signed measure defined in Assumption 1.
  Then there is a unique solution $f$ of (\ref{eq:I}) in the class of Stieltjes transforms of the non-negative measures. Moreover,
  $\lim_{\eta\to \infty}\eta|f(i\eta)|=1$, so that the corresponding to $f$ measure $N$ is a probability measure, $N(\R)=1$.
\end{lemma}
\begin{proof}
  We show first that if $f$ is the Stieltjes transform of a non-negative measure $N$ then for any $\xi\in\R$
  \begin{equation}\label{bound}
    |1+\xi f(z)|^{-1}\le\max\{2, \, 4|\xi|/|\Im z|\}.
  \end{equation}
  (Note also that for $\xi>0$ we have a simpler bound $|1+\xi f(z)|^{-1}\le |z|/|\Im z|$, which follows from the inequality $\Im z \Im f(z)\ge 0$.) To this end given $\xi\in \R$ define
  $$
  E_\xi:=\big\{z:\, |\xi|| f(z)|=|\xi|\Big|\int(\lambda-z)^{-1}{dN(\lambda)}\Big|<{1}/{2}\big\}.
  $$
  If $z\in E_\xi$, we have $|1+\xi f(z)|>1/2$. If $z\notin E_\xi$,  by the Schwartz inequality
  $\int |\lambda-z|^{-2}dN\ge{1}/(2|\xi|)^2\ $, so that
  $
  |1+\xi f(z)|\ge|\xi||\Im z|\int |\lambda-z|^{-2}dN\ge |\Im z|/(4|\xi|),
  $
  and (\ref{bound}) follows.

  By the conditions of the lemma we have
  \begin{equation}\label{cp}
    \int d\ws(\xi)=c_1<\infty\quad\text{and}\quad \int |\xi^p d\ws(\xi)|<\infty,\,\,p=1,2,3.
  \end{equation}
  Note that $\int |d\ws(\xi)|$ is not necessarily finite, that is why it is better to rewrite (\ref{eq:I}) in the form
  \begin{equation*}
    zf(z)=-1+c_1f(z)-f(z)^2\int\frac{\xi d\ws(\xi)}{1+\xi f(z)},
  \end{equation*}
  where now by (\ref{bound}) -- (\ref{cp}),  $\int|\xi (1+\xi f(z))^{-1} d\ws(\xi)|$ is uniformly bounded in
  \begin{equation}\label{C0}
    z\in\C_{\eta_0}:=\{z\in \C:\,\eta=\Im z \geq \eta_0\}
  \end{equation}
  for some  $\eta_0>0$. In particular  this allows to show that $\lim_{\eta\to \infty}\eta|f(i\eta)|=1$. Next, if there are two solutions $f_1$, $f_2$ of this equation, than
  $$
  z(f_1-f_2)=(f_1-f_2) \Big(c_1-\int\frac{\xi(f_1+f_2+\xi f_1f_2) d\ws(\xi)}{(1+\xi f_1)(1+\xi f_2)}\Big),
  $$
  where as it follows from (\ref{bound}) -- (\ref{cp}), if $f_1\neq f_2$  then the r.h.s. is uniformly bounded and the l.h.s. tends to infinity as $z \to\infty$. Hence $f_1=f_2$. The solvability of (\ref{eq:I})  in the class of Stieltjes transforms of the non-negative measures follows from the Banach fixed-point theorem.
\end{proof}

Let $G(z):=(\mathcal{L}_n   -zI_n)^{-1}$, $z\in \C\setminus\R$,
 be the resolvent of $\mathcal{L}_n $, so that
$s_n=n^{-1}\Tr G$. Lemma \ref{l:var} below shows that the variance of    $s_n$ tends to zero as $n\to\infty$, hence, by Chebyshev's inequality the convergence of $(s_n)_n$ in probability  follows from the convergence in means.

\begin{lemma}\label{l:var}(Self-averaging properties.)
Under conditions of Theorem \ref{t:1} we have uniformly in $z\in\C_{\eta_0}$ for big enough $\eta_0$
\begin{align}
   &(i)\,\, \V \,n^{-1}\Tr G(z)=O(n^{-1}),\quad\text{and}\label{TrG}
   \\
   &(ii)\,\,\sup_\beta\V\Tr \Qb G(z)=O(n^{-1}),\,\, n\to\infty. \label{TrQG}
   \end{align}
 \end{lemma}
 \begin{proof} Our proof is based on a  standard martingale technique introduced in random matrix theory by Girko (see \cite{Dh-Co:68} and \cite{Pa-Sh:11}).  For every $1\leq \al\leq m$, introduce
\begin{align*}
\mathcal{L}_n   ^\al=\mathcal{L}_n   -\xi_\al \ya  \ya ^T \quad\text{and}\quad   G^\al(z)=(\mathcal{L}_n   ^\al-zI_n)^{-1},
\end{align*}
so that $\mathcal{L}_n   ^{\al}$, $G_n^{\al}$ do not depend on $\xi_{\al}$ and $\ya $.
 Applying the result of \cite{Dh-Co:68},  one can get
  \begin{align*}
   \V \,n^{-1}\Tr G(z) &\le \frac{1}{n^2} \sum_{\al}\E\big|\Tr (G -\E_{\al}G )\big|^2
    \le\frac{4}{n^2} \sum_{\al}\E\big|\Tr (G -G^{\al} )\big|^2,
        \end{align*}
     (see also Lemma 3.2 \cite{Ly:17}), where by the resolvent identity
     $$
     \E\big|\Tr (G -G^{\al} )\big|^2=\E\big|\xi_\al(G^{\al}G\ya,\ya)\big|^2\le\E|\xi_\al|^2\E\|\ya\|_2^4/\eta_0^4,
     $$
  and (\ref{TrG}) follows.  Here we also used that as it follows from Assumptions 1,2, $\E|\xi_\al|^2$ and $\E\|\ya\|_2^4$ are bounded. Similarly we have
   \begin{align}\label{Tr2}
   \V\Tr \Qb G(z) &\le 4 \sum_{\al}\E\big|\Tr \Qb (G -G^{\al} )\big|^2\le 4 \sum_{\al}\E|\xi_\al|^2\E\|\ya\|_2^4\|\Qb\|^2_{op} /\eta_0^4,
        \end{align}
  and by (\ref{Qal}) we get (\ref{TrQG}).
  \end{proof}
  \medskip

   As it follows from Lemma \ref{l:var} (i), it remains to show that for every $z\in \C\setminus\R$ the expectations of $s_n$ converge to $f$ which solves (\ref{eq:I}). Then since $f$ is the Stieltjes transform of a non-negative measure (which is in fact a probability measure due to the tightness of $\E\NN_{\mathcal{L}_n   }$), Lemma \ref{l:solv} finishes the proof of  Theorem \ref{t:1}.
  \begin{lemma}(Convergence in mean.)\label{l:mean}
    Let  $f_n:=\E s_n=n^{-1}\E\Tr G$. Then for every $z\in \C\setminus\R$ there exists $\lim_{n\to\infty}f_n(z)=:f(z)$, and $f$ satisfies (\ref{eq:I}).
  \end{lemma}
  \begin{proof}
    Since $|f_n(z)|\le|\Im z|^{-1}$, there is a subsequence $(f_{n_j})_{n_j}$ and an analytic function $f(z)$, $z\in \C\setminus\R$, such that  $(f_{n_j})_{n_j}$ converges to $f$ uniformly on every compact set in  $\C\setminus\R$.  Due to the uniqueness property of analytic functions it suﬃces to consider domain $\C_{\eta_0}$ (\ref{C0})
for some fixed $\eta_0>0 $ which will be chosen later, and to show that every convergent subsequence converges in $\C_{\eta_0}$ to a solution $f$ of (\ref{eq:I}).
\medskip

Saving notation $(f_n)_n$ for a convergent subsequence, applying the resolvent identity, $zG=-1+G\mathcal{L}_n$, and a rank-one perturbation formula
\begin{equation}\label{rankone}
  G-G^{\al}=-\frac{\xi_\al G^{\al}\ya \ya ^T G^{\al}}{1+\xi_{\al}(G^{\al}\ya ,\ya )},
\end{equation}
we get
\begin{align}\label{fn1}
zf_n+1&=\frac{1}{n}\sum_\al \E\xi_{\al}(G\ya ,\ya )\notag
=\frac{1}{n}\sum_\al \E\frac{\xi_{\al}(G^{\al}\ya ,\ya )}{1+\xi_{\al}(G^{\al}\ya ,\ya )}
\\
&=\int\frac{1}{m}\sum_\al \E\frac{(G^{\al}\ya ,\ya )}{1+\xi(G^{\al}\ya ,\ya )}d\ws_n(\xi).
\end{align}
By the conditions of the theorem we have
\begin{align}
  &\E (G^{\al}\ya ,\ya )=\E\Tr \Qa G^\al,\notag
  \\
  &\V_\al(G^{\al}\ya ,\ya )=o(1),\quad n\to\infty,\label{VGyy}
\end{align}
where in the first equality we can replace $G^\al$ with $G$. Indeed, by the resolvent identity and (\ref{Qal}) we have
\begin{align}\label{QaGa-G}
 |\E \Tr\, \Qa (G^\al-G)|= |\E\xi_\al (G \Qa  G^{\al}\ya,\ya)|\le \|\Qa \|_{op}|\xi|\Tr \Qa/\eta_0^{2}=O(n^{-1}),
\end{align}
(more precisely, $\int  |\E \Tr\, \Qa (G^\al-G)d\ws_n(\xi)|=o(1)$),  and we also used that
$\E\|\ya\|^2=\Tr \Qa.$ Hence, introducing
$$
\fa:=\E\Tr G \Qa ,
$$
 and applying (\ref{Q}) we get
 \begin{align*}
  &\E (G^{\al}\ya ,\ya )=\fa+O(n^{-1})\quad\text{and}\quad f_n=\frac{1}{m}\sum_\al \fa +o(1).
  \end{align*}
Using the above equalities we get
\begin{align*}
  \frac{1}{1+\xi(G^{\al}\ya ,\ya )}&=\frac{1}{1+\xi f_n}\Bigg(1+\frac{\xi(f_n-f_{n,\al})}{1+\xi(G^{\al}\ya ,\ya )}
  -\frac{\xi (G^{\al}\ya ,\ya )^\circ}{1+\xi(G^{\al}\ya ,\ya )}\Bigg)+o(1),
\end{align*}
where $x^\circ=x-\E x$. This and (\ref{fn1}) yield
\begin{align}\label{fn2}
zf_n+1&=\int\frac{f_n}{1+\xi f_n}d\ws_n(\xi)+R_n+R'_n+o(1),
\\
R_n&=\int\frac{1}{1+\xi f_n}\frac{1}{m}\sum_\al \E\frac{(G^{\al}\ya ,\ya )^\circ}{1+\xi(G^{\al}\ya ,\ya )}d\ws_n(\xi),\notag
\\
R'_n&=\int\frac{1}{1+\xi f_n}\frac{1}{m}\sum_\al \E\frac{(\fa-f_{n})}{1+\xi(G^{\al}\ya ,\ya )}d\ws_n(\xi).\notag
\end{align}
By the Schwartz inequality  $ \E|(G^{\al}\ya ,\ya )^\circ|\le (\V (G^{\al}\ya ,\ya ))^{1/2}$, where by (\ref{TrQG}) and (\ref{VGyy})
\begin{align}\label{VGyy}
\V (G^{\al}\ya ,\ya )=\E\V_{\al}(G^{\al}\ya ,\ya )+\V \Tr G^{\al} \Qa=o(1).
\end{align}
Note that by (\ref{bound}),
$$
|1+\xi f_n(z)|^{-1}\le\max\{2, \, 4|\xi|/|\Im z|\}.
$$
 Also it follows from (\ref{rankone}) that
$(1+\xi_{\al}(G^{\al}\ya ,\ya ))^{-1}=1-\xi_{\al}(G\ya ,\ya )$, hence,
\begin{equation*}
  \frac{1}{|1+\xi(G^{\al}\ya ,\ya )|}\le 1+|\xi|\|\ya\|^2_2/\eta_0.
\end{equation*}
This and Assumption 2 allow to get
\begin{align}
|R_n|&={o}(1)\quad\text{and}\notag
\\
|R'_n| &\leq \frac{C}{n}\sum\limits_\am  |\fa -f_n|\leq C\Delta_n(z),\quad\text{where}\quad
\Delta_n(z)=\max\limits_{\al} |f_{n,\al}(z)-f_n(z)|\label{rn=}
\end{align}
and $C >0$ depends only on $\eta_0$. To finish the proof it remains to show that $\Delta_n=o(1)$.
Repeating all the steps leading to (\ref{fn2}) -- (\ref{rn=}) one can get
\begin{align}\label{fn3}
z\fa+\Tr \Qa&=\sum_\al \E\frac{\xi_{\al}(\Qa G^{\al}\ya ,\ya )}{1+\xi_{\al}(G^{\al}\ya ,\ya )}
\\
&=\int\frac{n}{m}\sum_\al \E\frac{(\Qa G^{\al}\ya ,\ya )}{1+\xi(G^{\al}\ya ,\ya )}d\ws_n(\xi)\notag
\\
&=\int\frac{\fa}{1+\xi f_n}d\ws_n(\xi)+R_{n,\al}+R'_{n,\al}+o(1),\notag
\end{align}
where
\begin{align}
R_{n,\al}&=-\int\frac{1}{1+\xi f_n}\frac{n}{m}\sum_\be \E\frac{\xi(\Qa G^{\be}\yb ,\yb )(G^{\be}\yb ,\yb )^\circ}{1+\xi(G^{\be}\yb ,\yb )}d\ws_n(\xi)=o(1),\label{Rna}
\\
R'_{n,\al}&=-\int\frac{1}{1+\xi f_n}\frac{n}{m}\sum_\be \E\frac{\xi(\Qa G^{\be}\yb ,\yb )(f_{n,\be}-f_{n})}{1+\xi(G^{\be}\yb ,\yb )}d\ws_n(\xi),
\quad |R'_{n,\al}| \leq C\Delta_n(z),\label{rna}
\end{align}
and we used additionally that by (\ref{Qal}) $\|\Qa\|_{op}=O(n^{-1})$.
It follows from (\ref{fn2}) and (\ref{fn3}) that
\begin{align*}
z(\fa - f_n)+(1-\Tr\Qa)= (\fa-f_n)\int \frac{d\ws_n(\xi)}{1+\xi f_n}+R'_{n,\al}- R'_{n}+o(1),
\end{align*}
hence, using bounds for $R'_{n,\al}$ and $R'_{n}$ and also (\ref{Qal}) we get
\begin{align*}
\Big|z-\int \frac{ d\ws_n(\xi)}{1+\xi f_n}\Big||\fa - f_n|\le C\Delta_n(z)+{o}(1),
\end{align*}
where $C>0$ is uniformly bounded in $\eta_0$. Choosing $\eta_0$ big enough one can get
$$
\Big|z-\int \frac{ d\ws_n(\xi)}{1+\xi f_n}\Big|>2C,
$$
which implies $2C|\fa - f_n|\le C\Delta_n(z)+{o}(1)$. Taking the maximum
over $\al\le m$ we get
\begin{equation}\label{Del}
  \Delta_n(z)={o}(1),\quad n\to \infty.
\end{equation}
 This leads to $R'_n=o(1)$ as $n\to \infty$ and finishes the proofs of the lemma and of the theorem.
\end{proof}

  \section{Proof of Theorem \ref{t:2}}
\label{s:t2}

Given $r,d\in\N$, let $\LL$ be defined in (\ref{AL}) - (\ref{B}) and (\ref{LY}):
\begin{align*}
  &\LL=\sum_{1\le k < \ell\le r}\xi_{k\ell} Y^{k\ell}Y^{k\ell\,T},\quad
  Y^{k\ell}=(Y_j)_{j=1}^r=((\delta_{jk}-\delta_{j\ell})\vv^{kl})_{j=1}^{r}\in \R^{rd},
  \end{align*}
where $(\xi_{k\ell})_{1\le k<l\le r}$ are  iid  copies of a $0/1$ random variable $\xi=\xi_r$ with $\p(\xi=1)=p_r$, and $(\vv^{kl})_{1\le k<l\le r}$ are mutually independent normalized isotropic random vectors, $\E\vv^{kl}_\al\vv^{kl}_\beta=d^{-1}\delta_{\al\beta}$,  satisfying Assumption 1 and having norms uniformly bounded in $r$,
$$
\sup_{k,\ell}\|\vv^{k\ell}\|^2_2\le C_0
$$
for some $C_0>0$. Here for block vectors of the form $X=(X_j)_{j=1}^r=(X_{j\al})_{j,\al=1}^{r,d}$ we use Latin indexes to count blocks and Greek indexes to count entries within a block. Let
$$
Q^{k\ell}=\Big(\Qk_{i\gamma,j \be}\Big)_{i,j,\gamma,\be=1}^{r,d}:=\E Y^{k\ell}Y^{k\ell T}
=\Big(\E Y^{k\ell}_{i} Y^{k\ell}_{j}\Big)_{i,j=1}^{r}.
$$
By the definition of $Y^{k\ell}$,
\begin{align*}
    &Q^{k\ell}  =\frac{1}{d}\Big((\delta_{jk}-\delta_{j\ell})(\delta_{ik}-\delta_{i\ell})I_d\Big)_{i,j=1}^{r},
\end{align*}
so it has only four non-zero blocks (equal $d^{-1}I_d$).
To check the conditions of Theorem \ref{t:1} note first that now
  $$
  m=r(r-1)/2,\quad n=rd,\quad\text{so that}\quad \lim_{n\to\infty}\frac{m}{n}\E\xi=\lim_{r\to\infty}\frac{p_r r}{2d}=c /2=c_1.
  $$
   For any $rd\times rd$ block matrix $D=\big(D_{ij}\big)_{i,j=1}^r$ with $d\times d$ blocks $D_{ij}=(D_{i\al,j\be})_{\al,\be=1}^d$ we have
  $$
  \big(DY^{k\ell},Y^{k\ell}\big)=\big(\widetilde{D}\vv^{k\ell},\vv^{k\ell}\big),\quad\text{where}\quad
  \widetilde{D}=D_{kk}+D_{\ell\ell}-D_{k\ell}-D_{\ell k},
  $$
  so that (\ref{Ayy2}) for $Y^{k\ell}$ follows from (\ref{Ayy2}) for $\vv^{k\ell}$. Also
  it is easy to check that $\Tr Q^{k\ell}=2$ and
  $$
  Q_{rd}:=\frac{2}{r(r-1)}\sum_{1\le k<\ell\le r} Q^{k\ell}=\frac{2}{rd}{ I}_{rd}+B_{rd},\quad\text{where}\quad
  B_{rd}=-\frac{2}{r(r-1)d}\Big((1-\delta_{ij}){ I}_d\Big)_{i,j=1}^r,
  $$
  and $\|B_{rd}\|_{HS}=o((rd)^{-1/2})$, $r\to\infty$, thus (\ref{Q}) is fulfilled. The only condition of Theorem \ref{t:1} which is not fulfilled is the first part of (\ref{Qal}), namely, we have $\|Q^{k\ell}\|_{op}=O(d^{-1})$ (instead of $\|Q^{k\ell}\|_{op}=O((rd)^{-1})$). On the other hand, matrix $Q^{k\ell}$ is very sparse and has only four non-zero blocks.

  Hence we need to go through the proof of Theorem \ref{t:1} and check the places, where condition (\ref{Qal}) was used. There are three such places: Lemma \ref{l:var} (ii), (\ref{QaGa-G}), and (\ref{Rna}) -- (\ref{rna}).  As to Lemma \ref{l:var} (ii), we reprove it in Lemma \ref{l:var2} below. Now we recall the main steps of the proof of Lemma \ref{l:mean} and check (\ref{QaGa-G}) and (\ref{Rna}) -- (\ref{rna}).

  Similar to (\ref{fn1}), one can get
  $$
  zf_r+1=
  c_1-\frac{c_1}{r(r-1)/2}\sum_{1\le k<\ell\le r}\E\frac{1}{\Ak},\quad\text{where}\quad\Ak=1+(\Gk\Yk,\Yk),
  $$
  $\Gk=(\LL-\xi_{k\ell}\Yk Y^{k\ell T}-zI_{rd})^{-1}$. It is easy to show that
  \begin{equation*}
    |\Ak|^{-1},\,|\E\Ak|^{-1}\le 1/(1-2C_0/\eta_0).
  \end{equation*}
  We also have $\E(\Gk\Yk,\Yk)=\E\Tr \Qk\Gk$. Similar to (\ref{QaGa-G}), here we can replace $\Gk$ with $G$. Indeed, since  by the definition of $Q^{k\ell}$,
  \begin{equation}\label{QXY}
  (\Qk X,Y)=\frac{1}{d}\sum_{\gamma}(X_{k\gamma}-X_{\ell\gamma})(Y_{k\gamma}-Y_{\ell\gamma}),\quad \forall X,Y\in \R^{rd},
  \end{equation}
    we have
        \begin{align*}
     |\E\Tr \Qk(\Gk-G)|&=|\E\xi_{k\ell}( \Qk\Gk\Yk,\overline{G}\Yk)|
     \\
     &=|\E\xi_{k\ell}\frac{1}{d}\sum_{\ga}((\Gk\Yk)_{k\ga}-(\Gk\Yk)_{\ell\ga})((\overline{G}\Yk)_{k\ga}-(\overline{G}\Yk)_{\ell\ga})|
     \\
     &\le\frac{4}{d\eta_0^2}\E\|\Yk\|^2_2=O(d^{-1}).
    \end{align*}
    Thus
    $$
    \E(\Gk\Yk,\Yk)=f_{r,k\ell}+O(d^{-1}), \quad f_{r,k\ell}=\E\Tr \Qk G,
    $$
    and repeating the steps leading to (\ref{fn1}) -- (\ref{rna}), we get
        $$
    zf_r+1=\frac{2c_1f_r}{1+2f_r}+R'_r+o(1), \quad |R'_r|\le C\Delta_r(z),\quad \Delta_r=\max_{k,\ell}|f_{r,k\ell}-2f_r|,
    $$
    and
    \begin{align*}
       zf_{r,k\ell}+2&=\frac{2c_1f_{r,k\ell}}{1+2f_r}+R_{r,k\ell}+R'_{r,k\ell}+o(1),
       \\
       R_{r,k\ell}&=\frac{2c_1}{1+2f_r}\frac{d}{r}\sum_{1\le i<j\le r}\E\frac{(\Qk\Gi\Yi,\Yi)\Ai^\circ}{\Ai}
       \\
       R'_{r,k\ell}&=\frac{2c_1}{1+2f_r}\frac{d}{r}\sum_{1\le i<j\le r}\E\frac{(\Qk\Gi\Yi,\Yi)}{\Ai}(f_{r,ij}-2f_r).
    \end{align*}
   It follows from (\ref{QXY}) that for any $X\in \R^{rd}$
  $$
  (\Qk X,Y^{ij})=[\delta_{ik}+\delta_{i\ell}-\delta_{jk}-\delta_{j\ell}]\frac{1}{d}\sum_{\gamma}(X_{k\gamma}-X_{\ell\gamma})\vv_{\gamma}^{ij}.
  $$
Hence instead of the double sums over $i,j$ in the expressions above we have  single sums over $i$ or over $j$. This and the boundedness of
$\vv^{ij}$ and $\Ai$ allows to treat $R_{r,k\ell}$ and $R'_{r,k\ell}$ similar to (\ref{Rna}) -- (\ref{rna}) and then to show that $R_{n,\al}$, $R'_{n,\al}=o(1)$ and to get the equation for $f_{\mathcal{L}}=\lim f_r$ (see (\ref{fAfL})).
\medskip

It remains to prove
\begin{lemma}\label{l:var2}
   $\V \Tr \Qk G=o(1)$, $r\to\infty$.
\end{lemma}
\begin{proof}
 We have (see (\ref{Tr2}) and (\ref{QXY}))
   \begin{align*}
   \V \Tr \Qk G\le 4\sum_{i<j}\E|\Tr\Qk (G^{ij}-G)|^2=4\sum_{i<j}\E|\xi_{ij}(\Qk G^{ij}\Yi,\overline{G}\Yi)|^2
   \\
   =\frac{4}{d^2}\sum_{i<j}\E\Big|\xi_{ij}\sum_{\gamma}((G^{ij}\Yi)_{k\gamma}-(G^{ij}\Yi)_{\ell\gamma})
   ((\overline{G}\Yi)_{k\gamma}-(\overline{G}\Yi)_{\ell\gamma})\Big|^2,
      \end{align*}
 where
  \begin{align*}
   \frac{1}{d^2}\sum_{i<j}\E\Big|\xi_{ij}\sum_{\gamma}(G^{ij}\Yi)_{k\gamma}
   (\overline{G}\Yi)_{k\gamma}\Big|^2&\le
   \frac{1}{d^2}\sum_{i<j}\E\xi_{ij}^2\|G^{ij}\Yi\|^2_2\sum_{\gamma}
   |(\overline{G}\Yi)_{k\gamma}|^2
  \\
  &\le \frac{C_0}{\eta_0^2d^2}\E\sum_{\gamma}\sum_{i<j}\xi_{ij}
   |({G}\Yi)_{k\gamma}|^2
     \end{align*}
 and by the definition of $\LL$ and the resolvent identity,
 \begin{align*}
  \sum_{i<j}\xi_{ij}|({G}\Yi)_{k\gamma}|^2
   =\sum_{i<j}\xi_{ij}  (\overline{G}\Yi Y^{ijT}G)_{k\ga,k\ga}=
    (\overline{G}\LL G)_{k\ga,k\ga}
  =(\overline{G}(zG+I_{rd}))_{k\ga,k\ga}=O(1).
  \end{align*}
   This finishes the proof of the lemma.
 \end{proof}


\section{Adjacency matrices. Proof of Theorem \ref{t:A}}
\label{s:t4}

The scheme of the proof is essentially the same as in the case of Laplacian $\LL$. The main difference is that here for every vector $X^{k\ell}$ in the definition of $\A$ (see  (\ref{LY})) there are two terms containing this vector, $X^{k\ell}X^{\ell k T}$ and $X^{\ell k}X^{k \ell  T}$, so that in order to separate this vector from the rest we need to apply the rank one perturbation formula twice. Also it is convenient to consider first  a more general model without block structure. We have

\begin{theorem}\label{t:Ag}
Given $n,m\in\N$, consider an $n\times n$ matrix
$
\mathcal{A}_n =\sum_{\am}\xi_\al(\xa  \ya^T+\ya \xa^T),
$
where
\medskip

(i) ${(\xi_\al)}_\al$ are iid copies of a  0/1 random variable $\xi=\xi_n$ with $\p( \xi=1)=p_n$,\\

(ii) $\frac{m}{n}p_n \to c_1 >0$ as $n\to \infty$ (without loss of generality we assume that $\frac{m}{n}p_n \equiv c_1$),\\

(iii) $(\xa)_\al,\,(\ya)_\al \subset \R^n$ are two sets of mutually independent random vectors such that
$\|\xa\|_2^2\leq C_0$, $\|\ya\|^2_2\leq C_0$
for some $C_0>0$ and  for all deterministic matrices $D=D_n$ with $\|D\|_{op}=1$ we have (cf (\ref{Ayy2}))
    \begin{equation*}
   \sup_{\uu,\vv\in(\xa,\ya)_\al} \V(D\, \uu,\vv)={o}(1),\quad n\rightarrow\infty.
    \end{equation*}

(iv) matrices $Q^{x\al}:=E \xa\xa^T$, $Q^{y\al}:=E \ya\ya^T$, and $Q^{xy\al}:=E \xa\ya^T=Q^{yx\al T}$ 
have the operator norms of order $O(n^{-1})$ and
$$
\sup_{\Qa\in(Q^{x\al},Q^{y\al})_\al}|\Tr \Qa-1|= o(1), 
$$

(v) for every $n\times n$ matrices $K_1$, $K_2$ we have
$$\frac{1}{m}\sum_{\am}\Tr Q^{x\al}K_1\Tr Q^{y\al}K_2=\frac{1}{n}\Tr K_1\frac{1}{n}\Tr K_2,$$

(vi) matrix $Q^{xy}:=(\frac{1}{m}\sum_\am|Q_{ij}^{xy\al}|)_{i,j}$ satisfies $\|Q^{xy}\|_{HS}= o(n^{-1/2})$.\\

Then as $n\to\infty$   the empirical spectral distributions $\NN _{\mathcal{A}_n }$ converge in probability to a non-random probability measure $\NN _{\mathcal{A}}$  which  Stieltjes transform $f$ is uniquely determined by the equation
  \begin{equation}\label{eq:A}
    zf^3+(1-2c_1 )f^2-zf-1=0
  \end{equation}
in the class of Stieltjes transforms of non-negative measures.
\end{theorem}

\begin{remark}\label{r:A}
  A more general case corresponding to $\xi_\al$ satisfying Assumption 2 contains more pure technical details and we do not treat it here, but we strongly believe that following essentially  the same scheme one can prove that in this case $f$ solves the equation
   \begin{equation*}
    zf=-1-2f^2\int\frac{\xi d\ws(\xi)}{1-\xi^2 f^2}.
  \end{equation*}
  \end{remark}

\begin{proof}
Following the scheme of the proof of Theorem \ref{t:1} note first that the proof of the analog of Lemma~\ref{l:solv} is trivial in this case and the proof of the analog of Lemma~\ref{l:var} is essentially the same. Thus we only need to prove the convergence in mean (cf~Lemma~\ref{l:mean}). To this end introduce
$$
\mathcal{A}_n ^{\al}:=\mathcal{A}_n -\xi_\al(\xa  \ya^T+\ya \xa^T)\quad\text{and}\quad G^\al (z):=(\mathcal{A}_n ^\al-zI_n)^{-1},
$$
where
$ z \in \C_{\eta_0}$ for a big enough $\eta_0$.
Given an $n\times n$ symmetric matrix $K$, applying twice (\ref{rankone}) we get
\begin{equation}\label{ranktwo}
(KG\xa,\ya)=\frac{(KG^\al\xa,\ya)(1+\xi_\al(G^\al\xa,\ya))-\xi_\al^2(KG^\al\ya,\ya)(G^\al\xa,\xa)}{(1+\xi_{\al}(G^{\al}\xa ,\ya ))^2-\xi_{\al}^2(G^{\al}\ya ,\ya )(G^{\al}\xa ,\xa )}.
\end{equation}
It follows from the resolvent identity, (ii) and (\ref{ranktwo}) with $K=I$, that
\begin{align}\label{fn_1}
zf_n(z)+1&=\frac{2}{n}\sum_\al \E\xi_{\al}(G\xa ,\ya )\notag\\
&=\frac{2c_1}{m}\sum_\al \E\frac{(G^\al\xa,\ya)(1+(G^\al\xa,\ya))-(G^\al\ya,\ya)(G^\al\xa,\xa)}{(1+(G^{\al}\xa ,\ya ))^2-(G^{\al}\ya ,\ya )(G^{\al}\xa ,\xa )}\notag
\\
&=2c_1-\frac{2c_1}{m}\sum_\al \E\frac{1+(G^\al\xa,\ya)}{A_\al}\notag\\
&=2c_1-\frac{2c_1}{m}\sum_\al \frac{1+\E\Tr Q^{xy\al}G^\al}{\E A_\al} +R_n,
\end{align}
where
\begin{align*}
A_\al&=(1+(G^{\al}\xa ,\ya ))^2-(G^{\al}\ya ,\ya )(G^{\al}\xa ,\xa ),
\\
R_n&=\frac{2c_1}{m}\sum_\al \frac{1}{\E A_\al} \E\frac{(1+(G^\al\xa,\ya))A_\al^{\circ}}{A_\al}.
\end{align*}
 Applying (iii) and an analog of Lemma \ref{l:var2} (ii), one can show that
 \begin{align*}
  &|A_\al|\geq \frac{1}{2}-3C_0/\eta_0^2>0,
  \\
  &\V A_\al={o}(1), \ \text{(see also (\ref{VGyy})) and}
  \\
  &\E A_\al=(1+\E \Tr Q^{xy\al}G^{\al})^2-\E\Tr Q^{x\al}G^\al\E\Tr Q^{y\al}G^\al + o(1),
 \end{align*}
where with the help of (iv) $G^\al$ can be replaced with $G$ with an error term of order $O(n^{-1})$ (cf (\ref{QaGa-G})) so that
$$
\E A_\al=(1+\E \Tr Q^{xy\al}G)^2-f_n^{x\al}f_n^{y\al}+ o(1),
$$
and we use  notations
$$
f_n^{x\al}=\E\Tr Q^{x\al}G, \quad f_n^{y\al}=\E\Tr Q^{y\al}G.
$$
By (iv) and (vi), $|\Tr Q^{xy\al}G|=O(1)$ and
\begin{equation}\label{xy}
  \frac{1}{m}\sum_\am|\Tr Q^{xy\al}G|\leq \sum_{i,j}Q_{ij}^{xy}|G_{ij}|\leq\|Q^{xy}\|_{HS}\|G\|_{HS}= o(1).
\end{equation}
Hence $R_n={o}(1)$ and
\begin{align*}
zf_n+1&=2c_1-\frac{2c_1}{m}\sum_\am \frac{1}{1-f_n^{x\al}f_n^{y\al}}+ o(1).
\end{align*}
It follows from (v) that ${m}^{-1}\sum_\am f_n^{x\al}f_n^{y\al}=f_n^2$, hence,
\begin{align}
zf_n+1=2c_1-\frac{2c_1}{1-f_n^2}+R'_n+ o(1),\label{zfn_1}
\end{align}
where
\begin{align}
&R'_n=-\frac{1}{1-f_n^2}\frac{2c_1}{m}\sum_\am\frac{f_n^{x\al}f_n^{y\al}-f_n^2}{1-f_n^{x\al}f_n^{y\al}},\label{rnad}
\\
&|R'_n|\leq\frac{2c_1}{\eta_0(1-\eta_0^{-2})^2}\Delta_n, \quad \Delta_n:=\max_\am(|f_n^{x\al}-f_n|+|f_n^{y\al}-f_n|).\notag
\end{align}
It remains to show that $\Delta_n= o(1)$. To this end we treat similarly $f_n^{x\al}$ (and  $f_n^{y\al}$) and applying (\ref{ranktwo}) with $K=Q^{x\al}$ we get
\begin{align*}
zf_n^{x\al}+\Tr Q^{x\al}&=\frac{2c_1}{m}\sum_\bm \E\frac{(nQ^{x\al}G^\be\xb,\yb)(1+(G^\be\xb,\yb))-(nQ^{x\al}G^\be\yb,\yb)(G^\be\xb,\xb)}{A_\be}.
\end{align*}
Note that $\|nQ^{x\al}\|_{op}=O(1)$, hence repeating steps leading to (\ref{zfn_1}) -- (\ref{rnad}) and applying (v), one can get
\begin{align}
zf_n^{x\al}+\Tr Q^{x\al}&=-\frac{2c_1}{m}\sum_\bm \frac{\E n\Tr Q^{x\al}Q^{y\be}G\E \Tr Q^{x\be}G}{1-f_n^{x\be}f_n^{y\be}}+ o(1)\notag\\
&=-2c_1\frac{f_n^{x\al}f_n}{1-f_n^2}+R'_{n,\al}+ o(1),\label{zfnxa}
\end{align}
where
\begin{align*}
&R'_{n,\al}=-\frac{1}{1-f_n^2}\frac{2c_1}{m}\sum_\bm\frac{f_n^{x\be}\E n\Tr Q^{x\al}Q^{y\be}G(f_n^{x\al}f_n^{y\al}-f_n^2)}{1-f_n^{x\be}f_n^{y\be}},
\\
&|R'_{n,\al}|\leq\frac{2c_1C_0^2}{(1-\eta_0^{-2})^2}\Delta_n.
\end{align*}
Now subtracting (\ref{zfnxa}) from (\ref{zfn_1}) and using (iv) one can show that $\Delta_n= o(1)$ (cf(\ref{Del})). Thus
\begin{equation*}
zf_n+1=-\frac{2c_1f_n^2}{1-f_n^2}+ o(1),
\end{equation*}
which leads to (\ref{eq:A}) and finishes the proof of Theorem \ref{t:Ag}.
\end{proof}
\medskip

{\it Proof of Theorem \ref{t:A}.} Now we have
\begin{equation}\label{A}
     \A=\sum_{1\le k \neq \ell\le r}\xi_{k\ell} X^{k\ell}X^{\ell k\,T},\quad X^{k\ell}=(\delta_{jk}\vv^{k\ell})_{j=1}^{n},
\end{equation}
so that in terms of Theorem \ref{t:2}
$$
\sum_{\al}=\sum_{1\le k<\ell\le r},\quad m=r(r-1)/2,\,\,n=rd,\,\, c_1=c/2,\quad \xa=X^{k\ell},\,\ya=X^{\ell k},\,k<\ell,
$$
and the analogs of $Q^{x\al}$ and $Q^{xy\al}$ are given by
\begin{align*}
  &Q^{kk}:=\E X^{k\ell}X^{k\ell T}=d^{-1}(\delta_{ik}\delta_{jk}I_d)_{i_j=1}^r,\,\, \Tr Q^{kk}=1,\,\,\text{and}
  \\
  &Q^{k\ell}:=\E X^{k\ell}X^{\ell kT}=d^{-1}(\delta_{ik}\delta_{j\ell}I_d)_{i_j=1}^r.
\end{align*}
We suppose that $(\vv^{k\ell})_{k<\ell}$ have uniformly bounded norms, so let $C_0>0$ be such that $\|\vv^{kl}\|^2_2\le C_0$ for every $k<\ell$.

Checking the conditions of Theorem \ref{t:Ag}, note first that (iii) follows from the definition of $X^{k\ell}$ and conditions for $\vv^{k\ell}$. As to (v-vi), these conditions are not fulfilled with $\sum_\al=\sum_{k<\ell}$, but since in (\ref{A}) we have $\sum_{k\neq\ell}$, by the definitions of $Q^{k\ell}$ we get the following analogs of (v-vi):
\begin{align*}
&\text{(v')}\quad
  \frac{1}{r}\sum_{k} \Tr Q^{kk}K= \frac{1}{rd} \Tr K\quad\text{and}
  \\
  &\text{(vi')}\quad
  \|\widetilde{Q}^{xy}\|_{HS}=O\big((r\sqrt{d})^{-1}\big),\quad\text{where}\quad
\widetilde{Q}^{xy}_{i\ga,j\be}:=\frac{1}{r^2}\sum_{k, \ell} | Q^{k\ell}_{i\ga,j\be}|=\frac{1}{dr^2}\delta_{\ga \be}.
   \end{align*}
Thus again the only condition which is not fulfilled is the first part of (iv), because now $\|Q^{kk}\|_{op}$, $\|Q^{k\ell}\|_{op}=O(d^{-1})$ (instead of $O((rd)^{-1})$). On the other hand, matrices $(Q^{kk})_k$ are ``orthogonal'' up to normalisation:
\begin{equation}\label{ort}
Q^{kk}Q^{p\ell}=\frac{1}{d}\delta_{kp}Q^{k\ell},\quad\text{and also}\quad
 Q^{kk}X^{p\ell}=\frac{1}{d}\delta_{kp}X^{k\ell},
\end{equation}
thus in the corresponding places of the proof we have single sums instead of double sums. This allows to repeat the proof of Theorem \ref{t:Ag} with slight modifications and to get first
\begin{align}
  zf_r+1=\frac{1}{rd}\sum_{k\neq\ell}\E\xi_{k\ell}(G X^{k\ell},X^{\ell k })\label{fn4}
    =c-\frac{c}{r^2}\sum_{k\neq\ell}\E\frac{1+(G^{k\ell}X^{k\ell},X^{\ell k})}{A_{k\ell}},
      \end{align}
where
$$
A_{k\ell}=(1+(G^{k\ell}X^{k\ell},X^{\ell k}))^2-(G^{k\ell}X^{k\ell},X^{ k\ell})(G^{k\ell}X^{\ell k},X^{\ell k}).
$$
Since for any matrix $B$, $|(B\vv^{k\ell},\vv^{k\ell })|\le C_0 \|B\|_{op}$, we have
\begin{equation}\label{VarA}
  \V A_{k\ell}\le C(\eta_0)\max_{k,\ell}\{\V(G^{k\ell}X^{k\ell},X^{\ell k}),\V(G^{k\ell}X^{k\ell},X^{k\ell})\},
\end{equation}
where we use notation $C(\eta_0)$ for every positive function uniformly bounded in $\eta_0\to\infty$,
$$
C(\eta_0)=O(1),\quad \eta_0\to\infty.
$$
It follows from Assumption 1 and Lemma \ref{l:var3} below that
\begin{align}\label{GklGlk}
  &\V(G^{k\ell}X^{k\ell},X^{\ell k})=\E\V_{k\ell}(G^{k\ell}X^{k\ell},X^{\ell k})+\V\Tr Q^{k\ell}G^{k\ell}=o(1),\quad\text{and}\notag
  \\
  &\V(G^{k\ell}X^{k\ell},X^{k\ell})=\E\V_{k\ell}(G^{k\ell}X^{k\ell},X^{k\ell})+\V\Tr Q^{kk}G^{k\ell}=o(1),
\end{align}
(cf (\ref{VarA})), hence, $\V A_{k\ell}=o(1)$, $r\to\infty$.
Also similar to (\ref{xy}) one can show that the terms containing $\E(G^{k\ell}X^{k\ell},X^{\ell k})=\Tr Q^{k\ell}G^{k\ell}$ do not contribute to the limit. It follows from above that
\begin{align}
  zf_r+1&=c-\frac{c}{r^2}\sum_{k,\ell}\frac{1}{1-f_r^{kk}f_r^{ \ell\ell}}+o(1)\notag
  \\
  &=c-\frac{c}{1-f_r^2}+R'_r+o(1),\label{fn5}
\end{align}
where
$$
f_r^{kk}=\E\Tr Q^{kk}G=d^{-1}\sum_{\ga}\E G_{k\ga,k\ga},\quad \frac{1}{r}\sum_{k}f_r^{kk}=f_r,
$$
and
\begin{align*}
  &R'_r=-\frac{c}{1-f_r^2}\frac{1}{r^2}\sum_{k,\ell}\frac{f_r^{kk}f_r^{ \ell\ell}-f_r^2}{1-f_r^{kk}f_r^{ \ell\ell}},\quad
  |R'_r|\le  C(\eta_0)\Delta_r,\quad \Delta_r=\max_{\ell}|f_r-f_r^{ \ell\ell}|.
\end{align*}
Using (\ref{ort}) -- (\ref{fn4}), similar to (\ref{zfnxa}) one can get for every $q\le r$
\begin{align}
  zf^{qq}_r+1&=\sum_{k\neq\ell}\E\xi_{k\ell}(Q^{qq}G X^{k\ell},X^{\ell k })
  =\frac{1}{d}\sum_{k\neq q}\E\xi_{kq}(G X^{kq},X^{q k })\label{fn5}
  \\
  &=c-\frac{c}{r}\sum_{k\neq q}\frac{1}{1-f_r^{kk}f_r^{ qq}}+o(1)\notag
  =c-\frac{c}{1-f_r^{qq}f_r}+R^q_r+o(1),
\end{align}
where $|R^q_r|\le C(\eta_0) \Delta_r$. Hence,
\begin{align*}
  z(f_r-f^{qq}_r)=\frac{cf_r(f^{qq}_r-f_r)}{(1-f_r^{qq}f_r)(1-f_r^{2})}+R'_r-R^q_r+o(1),
\end{align*}
and choosing $\eta_0$ big enough we get similar to (\ref{Del}) $\Delta'_r=o(1)$, which leads to (\ref{fn2}).  To  finish the proof of Theorem \ref{t:2}, it remains to prove the following statement, which is an analog of Lemma \ref{l:var2} (see also Lemma \ref{l:var}).

\begin{lemma}\label{l:var3}
  Let $V_r:=\max_{1\le k,\ell\le r}\V \Tr Q^{k\ell}G$. Then $V_r=o(1)$ as $r\to\infty$.
\end{lemma}
\begin{proof} Note that the simple trick based on the resolvent identity, which we have used in the last line of the proof of Lemma \ref{l:var2} to get reed of the double sum over $i,j$, does not work here. So we will go another way.

  For every $q\le r$, it follows from (\ref{fn5}) that
  \begin{align}
    z\V \Tr Q^{qq}G=&z\E \Tr Q^{qq}G(\Tr Q^{qq}\overline{G})^\circ\notag
    \\
    &=\frac{1}{d}\sum_{k\neq q}\E\xi_{kq}(G X^{kq},X^{q k })(\Tr Q^{qq}\overline{G}^{kq})^\circ+R_r,\label{varqq}
  \end{align}
  where $x^\circ=x-\E x$ and
  $$
  R_r=\frac{1}{d}\sum_{k\neq q}\E\xi_{kq}(G X^{kq},X^{q k })(\Tr Q^{qq}(\overline{G}-\overline{G}^{kq}))^\circ.
  $$
  We have
  $$
  |\Tr Q^{qq}(\overline{G}-\overline{G}^{kq})|=\xi_{kq}|(Q^{qq}G^{kq}X^{kq},\overline{G}X^{qk})+(Q^{qq}G^{kq}X^{qk},\overline{G}X^{kq})|
  \le 2\|Q^{qq}\|_{op}\|\vv^{qk}\|_2^2/\eta_0^2=O(d^{-1}),
  $$
  hence, $R_r=O(d^{-1})$. Also, we have
  $$
  \V \Tr Q^{qq}G=\V \Tr Q^{qq}G^{kq}+O(d^{-1}).
  $$
  Applying (\ref{ranktwo}), one can continue (\ref{varqq}) and get similar to (\ref{fn4})
  \begin{align*}
   z\V \Tr Q^{qq}G &=-\frac{c}{r}\sum_{k\neq q}\E\frac{(1+(G^{kq}X^{kq},X^{q k}))(\Tr Q^{qq}\overline{G}^{kq})^\circ}{A_{kq}}
   \\
   &=-\frac{c}{r}\sum_{k\neq q}\frac{\E(G^{kq}X^{kq},X^{q k})(\Tr Q^{qq}\overline{G}^{kq})^\circ}{\E A_{kq}}
   \\
   &\quad+\frac{c}{r}\sum_{k\neq q}\frac{1}{\E A_{kq}}\E\frac{(1+(G^{kq}X^{kq},X^{q k}))(\Tr Q^{qq}\overline{G}^{kq})^\circ A_{kq}^\circ}{ A_{kq}}
   =:T^{(1)}_r+T^{(2)}_r.
  \end{align*}
  Since by the Schwartz inequality
  \begin{align*}
     |\E(G^{kq}X^{kq},X^{q k})(\Tr Q^{qq}\overline{G}^{kq})^\circ|&=|\E\Tr Q^{kq}{G}^{kq}(\Tr Q^{qq}\overline{G}^{kq})^\circ|
     \\
 & \le (\V \Tr Q^{kq}{G}^{kq})^{1/2}(\V \Tr Q^{qq}{G}^{kq})^{1/2}\le V_r+O(d^{-1}),
  \end{align*}
  we have $|T^{(1)}_r|\le C(\eta_0) V_r+O(d^{-1})$. It follows from (\ref{VarA}) -- (\ref{GklGlk}) that
  $\V A_{kq}\le V_r+o(1)$. This and the Schwartz inequality allows to get
  $$
 | T^{(2)}_r|\le C(\eta_0) V_r^{1/2}(\max_{k,q}\V A_{kq})^{1/2}\le C(\eta_0) V_r+o(1).
  $$
  Summarising we get from for every $q\le r$
  $$
 \eta_0 \V \Tr Q^{qq}G \le C(\eta_0) V_r+o(1).
  $$
  Similar, one can show that $ \eta_0 \V \Tr Q^{qk}G \le C(\eta_0) V_r+o(1)$ for every $k,q$. Hence, taking maximum over $k,q$, we get
  $ \eta_0 V _r \le C(\eta_0) V_r+o(1)$, where $C(\eta_0)$ remains bounded as $\eta_0\to \infty$. Thus choosing $\eta_0$ big enough we get
  $V_r=o(1)$ as $r\to\infty$. This finishes the proof of the lemma and the proof of Theorem~\ref{t:A}.
\end{proof}

\address

\end{document}